\documentclass[letterpaper,11pt,reqno]{amsart} 
\usepackage[portrait,margin=3cm]{geometry} 
\usepackage{mathrsfs,xfrac} 
\usepackage[colorlinks=true,linkcolor=blue,citecolor=blue,urlcolor=blue]{hyperref} 
\usepackage{amsmath,amssymb,amsthm,amsfonts,amsbsy,latexsym,dsfont,color,graphicx,enumitem}
\usepackage[numeric,initials,nobysame]{amsrefs} 
\usepackage[foot]{amsaddr}

\usepackage[textsize=tiny]{todonotes}
\usepackage{regexpatch}
\makeatletter
\xpatchcmd{\@todo}{\setkeys{todonotes}{#1}}{\setkeys{todonotes}{inline,#1}}{}{}
\makeatother

\newtheorem{thm}{Theorem}[section]
\newtheorem{lem}[thm]{Lemma}
\newtheorem{cor}[thm]{Corollary}
\newtheorem{prop}[thm]{Proposition}
\newtheorem{defn}[thm]{Definition}
\newtheorem{rem}[thm]{Remark}
\newtheorem{obs}[thm]{Observation}

\newtheorem{conj}[thm]{Conjecture}

\newtheorem{question}{Question}

\begin{document}
\title[Irreducibility of random polynomials]{On the probability of irreducibility of random polynomials with integer coefficients}
\author[Grigory Terlov]{Grigory Terlov}

\address{University of Illinois at Urbana--Champaign, 1409 W Green Street, Urbana, Illinois 61801}
\email{gterlov2@illinois.edu}


\begin{abstract}
In this article we study asymptotic behavior of the probability that a random monic polynomial with integer coefficients is irreducible over the integers. We consider the cases where the coefficients grow together with the degree of the random polynomials. Our main result is a generalization of a theorem proved by Konyagin in 1999. We also generalize Hilbert's Irreducibility Theorem and present an analog of this result with centered Binomial distributed coefficients.
\end{abstract}

\maketitle
\section*{Acknowledgement}\label{Acknowledgement}
I thank Philip Matchett Wood for the introduction to the problem, supervision, and many insightful conversations throughout the process of writing this paper. I also thank Qiang Wu for helpful comments on improving the presentation of the article.

While the majority of this work was done in 2017-2018 as an undergraduate project, during the preparation for the submission Bary-Soroker, Koukoulopoulos, and Kozma proved a more general result in \cite{BS_K_K}.

\section{Introduction}\label{Introduction}

We call a polynomial with integer coefficients \textit{irreducible} if it cannot be written as the product of at least two polynomials of smaller degrees with integer coefficients. The asymptotic behavior of the probability that a random polynomial is irreducible has been studied from different angels. For instance, classical Hilbert's Irreducibility Theorem considers polynomials of fixed degree and with integer coefficients uniformly distributed on $[-K,K]$ and states that the probability of such a polynomial being irreducible over $\mathbb{Z}$ tends to $1$ as $K \to \infty $. On the other hand, Odlyzko and Poonen \cite{odlyzko} as well as Konyagin \cite{Konyagin} consider monic polynomials with coefficients taking values $0$ or $1$ with probability $1/2$ independently and with the constant coefficient of $1$. Odlyzko and Poonen conjectured that these polynomials are irreducible with high probability as the degree tends to infinity. This conjecture was recently proven in \cite{Breullard} assuming Riemann Hypothesis (RH) for the Dedekind $\zeta$ function. Without the assumption of RH, the closest result is due to Konyagin \cite[Theorem 1]{Konyagin}, who proved that the probability that such a polynomial is irreducible is bounded from below by $\frac{c}{\log d}$, where $d$ is the degree of the polynomial, $c$ is a positive constant. This theorem follows from another result \cite[Theorem 2]{Konyagin} which states that it is unlikely for such polynomials to be divisible by a factor of a degree up to $\frac{cd}{\log d}$ for some constant $c$. In recent years another class of polynomials was considered in this context by Bary-Soroker and Kozma in \cite{Bary-Soroker_Kozma}. They studied monic polynomials with integer coefficients that are chosen independently and uniformly at random from $\{1,2,\dots,210\}$ and proved that such polynomials are irreducible asymptotically almost surely as the degree tends to infinity. Moreover, it seems that their approach can be extended to monic polynomials with integer coefficients that are chosen uniformly from $\{1,2,\ldots,A-1,A\}$, where $A$ is equal to a product of four distinct prime numbers.

We will endeavour to connect previously mentioned approaches as we investigate the behavior of the probability that polynomials are irreducible when both the range of coefficients and the degree grow simultaneously. More specifically we are interested in the case where the range of the coefficients is symmetric around $0$. To simplify notation throughout the article denote for $d,K \in \mathbb{N}$ the set $$\mathcal{P}_{d,K}:=\Big{\{} f(x): f(x)=x^d+a_{d-1}x^{d-1}+\ldots+a_1x+a_0, \text{ where } a_0 \neq 0 \text{ and } \forall i, a_i \in [-K,K] \cap \mathbb{Z} \Big{\}}.$$

We propose the following conjecture that extends the one from \cite{odlyzko}, it is also similar to a conjecture from \cite{BBBSWW}, but stated in a less technical way.
\begin{conj}\label{conj}
Let $f\in \mathcal{P}_{d,K}$ with independent and uniformly distributed coefficients then
\begin{equation*}
    \mathbb{P}(\,f\text{ is irreducible} \,) \to 1, \text{ as } d,K \to \infty.
\end{equation*}
\end{conj}
This conjecture follows from \cite[Theorem 1]{BS_K_K}.
\section{Main Results}\label{Main results}

Our main result is a generalized version of \cite[Theorem 2]{Konyagin}, which we state in the following way:

\begin{thm}\label{th:main}
Suppose $d \geq 2$. $ K =K(d)$ is a naturally valued function, $m_0=\frac{\sqrt{d}}{\log^2 d}$, and $f$ be a random polynomial from $\mathcal{P}_{d,K}$ with independent and uniformly distributed coefficients.
\begin{enumerate}
    \item If $K\leq d^{a-1}$ for some natural number $a>1$, then in the limit as $d \to\infty$, the probability that $f$ is divisible by at least one polynomial with integer coefficients of positive degree not exceeding $m_0$ is at most $O\left(\sqrt{\frac{K}{d}}\right)$.\label{th:mainPart1}
    \item If $cd\leq K\leq d^{a-1}$, where $c$ is some constant and $a>1$ is a natural number, then the probability that $f$ is divisible by at least  one polynomial with integer coefficients of positive degree not exceeding $m_0$ is at most $O\left(\frac{1}{K}\right)$.\label{th:mainPart2}
\end{enumerate}
\end{thm}

Theorem~\ref{th:main} directly implies the following corollaries.
\begin{cor} Suppose $d \geq 2$ is an integer, $K=K(d)$ is a natural valued function such that $K\leq O\left(\sqrt{d}\right)$ and $m_0=\frac{\sqrt{d}}{\log^2 d}$. Also let $f$ be a random polynomial from $\mathcal{P}_{d,K}$ with independent and uniformly distributed coefficients. Then, the probability that $f$ is divisible by at least one polynomial with integer coefficients of positive degree not exceeding $m_0$ is at most $O\left(\frac{1}{d^{1/4}}\right).$ 
\end{cor}
\begin{cor}\label{cor2.2} Suppose $d \geq 2$ is an integer, and $m_0=\frac{\sqrt{d}}{\log d^2}$. Also let $f$ be a random polynomial from $\mathcal{P}_{d,1}$ with independent and uniformly distributed coefficients. Then, the probability that $f$ is divisible by at least one polynomial with integer coefficients of positive degree not exceeding $m_0$ is at most $O\left(\frac{1}{\sqrt{d}}\right).$
\end{cor}
\begin{rem} \normalfont{Corollary \ref{cor2.2} states that the upper bound achieved by Konyagin in \cite{Konyagin} can be extended from polinomials with $\{0,1\}$ coefficients to ones with $\{-1,0,1\}$ coefficients.}
\end{rem}

One can allow $K$ to have more than polynomial growth in $d$, however one needs to settle for a smaller upperbound on the degree of the factor. We state the following theorem for $m_1:=\frac{\sqrt{d}}{\log^2 (Kd)}\leq m_0$.

\begin{thm}\label{th:main2}
Suppose $d \geq 2$, $K=K(d)$ is a naturally valued function, $m_1=\frac{\sqrt{d}}{\log^2 (Kd)}$, and $f$ be a random polynomial from $\mathcal{P}_{d,K}$ with independent and uniformly distributed coefficients.
\begin{enumerate}
    \item If $K\leq e^{d^{0.25}}/d$, then in the limit as $d \to\infty$, the probability that $f$ is divisible by at least  one polynomial with integer coefficients of positive degree not exceeding $m_1$ is at most $O\left(\sqrt{\frac{K}{d}}\right)$.
    \item If $c_2 d\leq K\leq e^{(d^{0.25}/b)}/d$, where $b=\exp\left(e^{1/\sqrt[3]{4c_1 }}/2\right)$ and $c_1$,$c_2$ are positive constants, then the probability that $f$ is divisible by at least  one polynomial with integer coefficients of positive degree not exceeding $m_1$ is at most $O\left(\frac{1}{K}\right)$.
\end{enumerate}
\end{thm}

This paper is organized as follow: Section~\ref{proofs} is dedicated to proof of Theorem~\ref{th:main} and Theorem~\ref{th:main2}. In Section~\ref{Hilbert th} we continue connecting the approaches by considering Hilbert's irreducibility theorem in different settings. In Theorem \ref{th3.1} we modify a proof of Hilbert's irreducibility theorem to allow the degree to grow as $d(K)= o\left(\log_2K-\log_2(1+\log K)\right)$ to make a small step towards the Conjecture~\ref{conj}, while in Theorem \ref{th3.2} we consider random polynomials with centered binomial coefficients.

%
%
\section{Proof of the main results}\label{proofs}
\begin{defn}
We define $n^{th}$ cyclotomic polynomial as 
\begin{equation*}
Q_n(z)=\prod_{1\leq j\leq n, gcd(j,n)=1}(z-e^{2\pi j/n}).
\end{equation*}
Also define $\phi(n)=\deg Q_n$.
\end{defn}
From Konyagin's work \cite[Page 336]{Konyagin}, it follows that the number of $\{0,1\}$ polynomials of degree $d$ that are divisible by a noncyclotomic polynomial of degree up to $\sqrt{d}/\log^2 d$ is significantly less in order than $2^d/\sqrt{d}.$ Thus, the bound on the number of polynomials divisible by a cylotomic polynomial is the limiting case in Konyagin's result \cite[Theorem~2]{Konyagin}. Following this idea, the proof of the Theorem \ref{th:main} is based on two parts, Proposition~\ref{pr2.1} that is concerned with the cyclotomic case and Proposition~\ref{pr2.2} the non-cyclotomic one. The difference between part~\ref{th:mainPart1} and part~\ref{th:mainPart2} of the Theorem \ref{th:main} will show itself only in the Proposition ~\ref{pr2.1}. In order to analyze divisibility by cyclotomic polynomials we use a notion of a random variable being $p$-bounded of exponent $r$ that was originally introduced in \cite{bourgain}. 

\begin{defn}\label{def: pbdd} Let $p,q \in \mathbb{R}_{+}$ such that $0 <q< p< 1$ and let $r \in \mathbb{Z}_{+}$. A $\mathbb{Z}$-valued random variable $a$  is called $p \textit{-bounded of exponent }r$ if there exists a $\mathbb{Z}$-valued symmetric random variable $\beta^{(\mu)}$ taking value $0$ with probability $1-\mu=p$ such that the following conditions holds:  
\begin{enumerate}
\item $\max_x \mathbb{P}(a = x) \leq p$ 
\item $q \leq \min_x \mathbb{P}(\beta^{(\mu)}=x) \text{ and } \max_x \mathbb{P}(\beta^{(\mu)}=x) \leq p$
\item $\forall t \in \mathbb{R} \text{ we have } \left|\mathbb{E}(\exp(2\pi i a t))\right|^r \leq \mathbb{E}\left(\exp(2\pi i \beta^{(\mu)} t)\right).$ 
\end{enumerate}
\end{defn}
The third condition can be rewritten as 
\begin{equation*}
\left|\mathbb{E}(\exp(2\pi i a t))\right|^r \leq \mathbb{E}\left( \exp\left(2\pi i \beta^{(\mu)} t \right)\right)=1-\mu+\mu \sum_{s} p_s\cos(2\pi b_s t),    
\end{equation*} 
where $p_s= \mathbb{P}(\beta^{(\mu)}= b_s)/ \mu=\mathbb{P}(\beta^{(\mu)}= -b_s)/ \mu$.

\begin{lem}\label{lem: pbdd} Let an integer random variable $A$ take values uniformly on $[-K,K],$ for natural valued $K$. Then $A$ is $\left(1-\frac{4K}{(2K+1)^2}\right)$-bounded of the exponent 2 with coefficient $q= \frac{2K}{(2K+1)^2}$.
\end{lem}
\begin{lem}\label{lem: LWO} If $ \{ a_k \}^m_{k=1}$ is a collection of independent identical random variables that are $p$-bounded of the exponent $r$ with coefficient $q$, 
then \begin{equation*}
\mathbb{P}\left(\sum^m_{k=0}{a_k}=x\right)\leq \frac{C\sqrt{r}}{\sqrt{qm}}.
\end{equation*}
\end{lem}
We present the proofs of Lemmas~\ref{lem: pbdd} and~\ref{lem: LWO} in the Appendix. 
\begin{rem}The proof of Lemma~\ref{lem: LWO} is similar to the proof of \cite[Corollary 7.13]{Tao} and \cite[Lemma A1]{bourgain} in the case when all $v_i$ in the formulation of \cite[Lemma A.1]{bourgain} are equal to $1,$ however it differs in one important detail, namely we allow parameters $q$ and $m$ to change. It is also worth noticing that the proof of \cite[Lemma A.1]{bourgain} could work when $q$ is decreasing, however it was not needed in the context of \cite{bourgain} and, thus, Bourgain at al. decided to fix the parameter.
\end{rem}

To prove the part \ref{th:mainPart2} of Theorem~\ref{th:main}, and Proposition \ref{pr2.1}, we would need to improve the bound from Lemma~\ref{lem: LWO}. We do so in Lemma~\ref{lem2.2}. First we state a simple trigonometric Observation~\ref{trig_fact} and Lemma~\ref{lem2.1} that will be used in the future.
\begin{obs}\label{trig_fact} Suppose $j_1,j_2,\dots,j_\ell \in [1,K]$ for some $K\in\mathbb{N},$ then
\begin{equation*}
\prod_{i=1}^\ell\cos(2\pi t j_i)=\frac{1}{2^{\ell-1}}\sum_{s_2,\dots,s_\ell\in\{-1,1\}}\cos\left(2\pi t(j_1+s_2j_2+s_3j_3+\dots+s_\ell j_\ell)\right).
\end{equation*}
\end{obs}
\begin{proof}
This fact can be proved by induction and product to sum trigonometric identity.
\end{proof}

\begin{lem}\label{lem2.1}
For any natural numbers $\ell$ and $K$, the following inequality holds
\begin{equation*}
\int_0^1\left(\sum_{j=1}^K\cos(2\pi jt)\right)^\ell dt\leq\left(1-{2^{-(\ell-1)}}\right)K^{\ell-1}.
\end{equation*}
\end{lem}
\begin{proof}
\begin{align*}
    \int_0^1\left(\sum_{j=1}^K\cos(2\pi jt)\right)^\ell dt &=\int_0^1\left(\sum_{j_1=1}^K\cos(2\pi j_1t)\right)\left(\sum_{j_2=1}^K\cos(2\pi j_2t)\right)\cdots\left(\sum_{j_\ell=1}^K\cos(2\pi j_\ell t)\right) dt\\
    &=\int_0^1\left[\sum_{1\leq j_1,j_2,\dots,j_\ell\leq K}\cos(2\pi j_1 t)\cos(2\pi j_2t)\cdots\cos(2\pi j_\ell t)\right]dt.\notag
\end{align*}
Using the trigonometric Observation \ref{trig_fact}, we can rewrite the above as
\begin{equation}\label{eq:sum prod}
     =\int_0^1\left[\sum_{1\leq j_1,j_2,\dots,j_\ell\leq K}\frac{1}{2^{\ell-1}}\sum_{s_2,s_3,\dots,s_\ell\in\{-1,1\}}\cos(2\pi t(j_1+s_2 j_2+s_3j_3+\dots+s_\ell j_{\ell} ))\right]dt\\
\end{equation}
\begin{equation*}\label{eq:sum prod2}
     =\sum_{1\leq j_1,j_2,\dots,j_\ell\leq K}\frac{1}{2^{\ell-1}}\sum_{s_2,s_3,\dots,s_\ell\in\{-1,1\}}\int_0^1\cos(2\pi t(j_1+s_2 j_2+s_3j_3+\dots+s_\ell j_{\ell} ))dt.\\
\end{equation*}
Notice that all terms in the second sum are of the from $\cos(2\pi x t)$ where $x\in\{-\ell K,\dots,-1,0,1,\dots\ell K\}$. Thus, the only terms that do not integrate to zero are those that have $x=0$, each of which integrates to $1$. Therefore, \eqref{eq:sum prod} is equal to $\frac{1}{2^{\ell-1}}(\#\text{ number of terms such that } j_1+s_2 j_2+s_3j_3+\dots+s_\ell j_{\ell} =0).$
For fixed $j_2,j_3,\dots,j_\ell$ and $s_2,s_3,\dots,s_\ell$, there can be at most one $j_1$ such that $j_1+s_2 j_2+s_3j_3+\dots+s_\ell j_{\ell} =0.$ Moreover, there is a case where the sum cannot be zero, namely when all $s_i=1$, since all $j_i\in{1,2,\dots,K}$. Thus, the 
$(\#\text{ number of terms such that } j_1+s_2 j_2+s_3j_3+\dots+s_\ell j_{\ell} =0)\leq(2^{\ell-1}-1)K^{\ell-1}.$
\begin{equation*}
\int_0^1\left(\sum_{j=1}^K\cos(2\pi jt)\right)^\ell dt\leq\frac{2^{\ell-1}-1}{2^{\ell-1}}K^{\ell-1}=\left(1-{2^{-(\ell-1)}}\right)K^{\ell-1}.
\end{equation*}
\end{proof}
\begin{lem}\label{lem2.2}
Let $m\geq 2$ and $\{ a_k \}^m_{k=1}$ be a collection of i.i.d. symmetric  random variables that are $\left(1-\frac{4K}{(2K+1)^2}\right)$-bounded of the exponent $2$ with coefficient $\frac{2K}{(2K+1)^2}$,
then 
\begin{equation*}
\mathbb{P}\left(\sum^m_{k=0}{a_k}=x\right) \leq \left(\frac{1}{2K+1}\right)^m+\frac2K.
\end{equation*}
\begin{proof}
Suppose $Q$ is a sufficiently large prime number.
\begin{align*}
\mathbb{P}\left(\sum^m_{k=0}{a_k}=x\right)&=\mathbb{E}\mathds{1}_{\{\sum^m_{k=0}{a_k}=x\}}\\
&=\frac{1}{Q}\mathbb{E}\sum_{\xi \in \mathbb{Z}/Q\mathbb{Z}}\exp\left(2\pi i \left(\sum^m_{k=0}{a_k}-x\right)\xi /Q\right) \notag\\
&\leq \frac{1}{Q}\sum_{\xi \in \mathbb{Z}/Q\mathbb{Z}}\prod_{k=0}^{m} \left|\mathbb{E}\exp(2\pi i a_k\xi /Q)\right| \notag\\
&\leq \prod_{k=0}^{m}\left(\frac{1}{Q}\sum_{\xi \in \mathbb{Z}/Q\mathbb{Z}} |\mathbb{E}\exp(2\pi i a_k\xi /Q)|^m\right)^{1/m} & \mbox{ (by H\"older's inequality).}\notag
\end{align*}
Since $m\geq 2$, without loss of generality, we can assume that $m$ is even and ignore the absolute value. Now by passing to the largest factor,
\noindent
\begin{align*}
    &\leq \frac{1}{Q}\left(\sum_{\xi \in \mathbb{Z}/Q\mathbb{Z}} |\mathbb{E}\exp(2\pi i a_{k}\xi /Q)|^m\right) \text{ by symmetry of } a_k\\
    &= \frac{1}{Q}\left(\sum_{\xi \in \mathbb{Z}/Q\mathbb{Z}} \left|\frac{1}{2K+1}+\frac{2}{2K+1}\sum_{j=1}^K\cos(2\pi jt)\right|^m\right) \text{ letting }  a=\frac{1}{2K+1}, b=\frac{2}{2K+1}\notag\\
    &\leq \frac{1}{Q}\sum_{\xi \in \mathbb{Z}/Q\mathbb{Z}} \left( a^m+\sum^{m}_{\ell=1}\left(\binom{m}{\ell}a^{m-\ell}\left(b\sum_{j=1}^K\cos(2\pi j t)\right)^\ell\right)\right).\notag
\end{align*}
Here we pass to integral form and choose $Q$ large enough so that the error is at most $1/K$.
\begin{align*}
    &\leq \int_0^1 \left[a^m+\sum^{m}_{\ell=1}\binom{m}{\ell}a^{m-\ell}\left(b\sum_{j=1}^K\cos(2\pi j t)\right)^\ell \right]dt + \frac{1}{K}\\
    &= a^m+\int_0^1 \left[\sum^{m}_{\ell=1}\left(\binom{m}{\ell}a^{m-\ell}\left(b\sum_{j=1}^K\cos(2\pi j t)\right)^\ell\right)\right]dt+\frac{1}{K}\notag\\
    &= a^m+\sum^{m}_{\ell=1}\left({m \choose \ell}a^{m-\ell}b^\ell\int_0^1 \left[\left(\sum_{j=1}^K\cos(2\pi j t)\right)^\ell\right]dt\right)+\frac{1}{K}.
\end{align*}
From Lemma \ref{lem2.1} we know the upper bound for the integral in the previous line, so
\begin{align*}
    &\leq a^m+\sum^{m}_{\ell=1}\left({m \choose \ell}a^{m-\ell}b^\ell\left(1-2^{-(\ell-1)}\right)K^{\ell-1}\right)+\frac1K\\
    &\leq\left(\frac{1}{2K+1}\right)^m+\frac{1}{K}\sum^m_{\ell=1}2^\ell K^\ell \binom{m}{\ell}\left(\frac{1}{2K+1}\right)^m+\frac1K\notag\\
    &\leq \left(\frac{1}{2K+1}\right)^m+\frac1K\left(\frac{1}{2K+1}\right)^m(2K+1)^m+\frac1K\notag\\
    &\leq \left(\frac{1}{2K+1}\right)^m+\frac2K.\notag
\end{align*}
\end{proof}
\end{lem}

Let $C$ and $c$ be some constants that may change from line to line.
\begin{prop}\label{pr2.1} Suppose $d \geq 2$, $a\in\mathbb{N}$, and $a>1$. $ K =K(d)\leq d^{a-1}$ is a naturally valued function, $m_0=\frac{\sqrt{d}}{\log^2 d}$, and $f$ be a random polynomial from $\mathcal{P}_{d,K}$ with independent and uniformly distributed coefficients. 
\begin{enumerate}
\item  The probability that $f$ is divisible by at least one cyclotomic polynomial of positive degree not exceeding $m_0$ is at most $O\left(\sqrt{\frac{K}{d}}\right)$ as $d\to \infty$,
\item if $K>cd$, for some constant $c$, then the probability that $f$ is divisible by at least one cyclotomic polynomial of positive degree not exceeding $m_0$ is at most $O\left(\frac{1}{K}\right)$ as $d\to \infty$.
\end{enumerate}
\end{prop}
\begin{proof}
Denote by $N_{n}$ the probability of the event that a random polynomial $f \in \mathcal{P}_{d,K}$ is divisible by $Q_n$. \newline
Thus, $N_{1}$ and $N_{2}$ are the probability that a polynomial $f \in \mathcal{P}_{d,K}$ is divisible by $(x-1)$ or $(x+1)$. Note that a random polynomial $f$ is divisible by $(x-1)$ or $(x+1)$ if and only if $f(1)=0$ or $f(-1)=0$ respectively.
Let us first estimate the probability of being divisible by $(x-1)$. In other words, we need to calculate the probability that a sum of coefficients $a_k$ adds up to $0$. From Lemma \ref{lem: pbdd} it follows that, $a_k$ is $\Big{(}1-\frac{4K}{(2K+1)^2}+\epsilon\Big{)}$-bounded of exponent 2 with coefficient $q= \frac{2K}{(2K+1)^2}$ for all $k\in\{0,1,\dots d\}$. From Lemma \ref{lem: LWO}, for the first statement, and Lemma \ref{lem2.2}, for the second statement, it follows that 
\begin{equation*}
\mathbb{P}\left(\sum^d_{k=0}{a_k}=0\right)\leq C \sqrt{\frac{(2K+1)^2}{Kd}} \text{    and   } \mathbb{P}\left(\sum^d_{k=0}{a_k}=0\right)\leq\frac{C}{K}    
\end{equation*}
Since $\{a_k\}$ are symmetric random variables, thus, $\{-a_k\}$ are also $\Big{(}1-\frac{4K}{(2K+1)^2}+\epsilon\Big{)}$-bounded of exponent 2 with coefficient $q= \frac{2K}{(2K+1)^2}$. Therefore, plugging them into the spots that correspond to even degrees of $x$ would imply that a similar argument works to estimate the probability of the event where $f$ is divisible by $(x+1).$
Thus, 
\begin{equation*}
N_{1}+N_{2}\leq O\left(\sqrt{\frac{K}{d}}\right)\text{    and   } N_{1}+N_{2}\leq O\left(\frac{1}{K}\right).
\end{equation*}
For $n \geq 3$ let 
 $f(z)=\sum_{j=0}^d a_jz^j$ 
be a polynomial in $\mathcal{P}_{d,K}$ and let 
$h(z)=\sum_{j=0}^{n-1} A_jz^j$,
where 
\begin{equation}\label{eq: A_j}
A_j=\sum_{k\equiv j (\textrm{mod}\, n)} a_j.
\end{equation}

In the next step we are following \cite[Section 4]{Konyagin}. The  polynomials $f$ and $h$ are congruent $\mod(z^n-1)$. Thus, $h$ is also divisible by $Q_n$(z). By \cite[Theorem 5.1]{Prachar} \begin{equation*}
n<C\phi(n)\log\log(\phi(n)+2).
\end{equation*} 
The coefficients $A_j$ are determined by the divisibility of $h$ by $Q_n$. Thus from the definition of $A_j$ it implies that all sums of fixed $a_k (k\equiv j (\textrm{mod }{n})) \phi(n)\leq j<n)$ are also determined. Fixing these coefficients also determines the other $a_k$, due to divisibility of $h$ by $Q_n.$ From Lemma \ref{lem: pbdd} it follows that these $a_k$ are $\Big{(}1-\frac{4K}{(2K+1)^2}+\epsilon\Big{)}$-bounded of exponent $2$ with coefficient $q= \frac{2K}{(2K+1)^2}$. \newline

Therefore, to prove the first statement of Proposition \ref{pr2.1}, we apply Lemma \ref{lem: LWO} with $m=d/n$, which gives us an upper bound of $C \sqrt{2n(2K+1)^2/2Kd}$ on the proportion of vectors $(a_j,a_{j+n},\dots,a_{j+[(d-j)/n]})$ to satisfy \eqref{eq: A_j}.
As in \cite{Konyagin} we notice that
\begin{equation*}
C \sqrt{2n(2K+1)^2/2Kd} \leq 2C \sqrt{ K m_0\log\log(m_0+2)/d} \leq \frac{C \sqrt{K}}{d^{1/4}}.
\end{equation*}
So $N_n\leq C\left(\frac{K}{\sqrt{d}}\right)^{l/2}$ where $l=\phi(n)>2$, therefore
\begin{align*}
&\mathbb{P} (f \in \mathcal{P}_{d,K} \text{ is divisible by at least one cyclotomic polynomial of degree up to } m_0 )\\
&\qquad\leq N_{1}+N_{2}+ C\left(\sum^{m_0}_{l=3} \#\{n:\phi(n)=l\}\left(\frac{K}{\sqrt{d}}\right)^{l/2}\right)\notag\\
&\qquad\leq O\left(\sqrt{\frac{K}{d}}\right)+\left(\sum^{m_0}_{l=3} \frac{ 2Cl \sqrt{\log\log(l+2)}\sqrt{K^l}}{d^{l/4}}\right)\notag\\
&\qquad\leq O\left(\sqrt{\frac{K}{d}}\right).\notag
\end{align*}

Thus,
the probability that $f$ is divisible by at least one cyclotomic polynomial of positive degree not exceeding $m_0$ is at most $O \left(\sqrt{\frac{K}{d}}\right).$

Finally, the proof of the second inequality is similar to the one above, however we apply Lemma \ref{lem2.2} with setting $m=d/n$ instead of Lemma \ref{lem: LWO} to get an upper bound of $C/K$ on the probability of vectors $\left(a_j,a_{j+n},\dots,a_{j+[(d-j)/n]}\right)$ to satisfy \eqref{eq: A_j}.
Same as above we notice that for $K$ large enough we have that
$N_n\leq O\left(K^{-l}\right)$ where $l=\phi(n)>2$
Therefore,
\begin{align*}
&\mathbb{P} \left(f \in \mathcal{P}_{d,K} \text{ is divisible by at least one cyclotomic polynomial} \right)\\
&\qquad\leq N_{1}+N_{2}+\left(\sum^{m_0(d)}_{l=3} \frac{\#\{n:\phi(n)=l\}}{K^{l}}\right)\notag\\
&\qquad\leq O\left(\frac{1}{K}\right)+\left(\sum^{m_0(K)}_{l=3} \frac{ 2Cl\sqrt{\log\log(l+2)}}{K^{l}}\right)\notag\\
&\qquad\leq O\left(\frac{1}{K}\right).\notag
\end{align*}

So as long as $K>cd$ the probability that $f$ is divisible by at least one cyclotomic polynomial of positive degree not exceeding $m_0$ is at most $O\left(\frac{1}{K}\right).$
\end{proof}

\begin{question}[Factors of low degree]
The proof of Proposition~\ref{pr2.1} suggests that most of the contribution to the probability that $f$ is divisible by at least one cyclotomic polynomial come from being divisible by $N_1$ and $N_2$. This was established in \cite{lowdeg}. Thus, it is of interest to determine what is the next likeliest cyclotomic polynomial to divide $f$. One could also investigate more carefully an upper bound of the probability of $f$ to be divisible by a non-linear cyclotomic polynomial of a degree not exceeding $m_0$, and possibly improve our bound.
\end{question}
\begin{prop} \label{pr2.2} Suppose $d \geq 2$, $K =K(d)\leq d^{a-1}$ is a naturally valued function, and $m_0=\frac{\sqrt{d}}{\log^2 d}$. Also let $f$ be a random polynomial from $\mathcal{P}_{d,K}$ with independent and uniformly distributed coefficients. Then, the probability of $f$ being divisible by a noncyclotomic polynomial of positive degree not exceeding $m_0$ is at most $O\left((2K+1)^{-\frac{ Cd}{(\log d)^4}}\right)$ for some positive constant $C$.
\end{prop}
\begin{proof}
This proof is modeled on Konyagin's proof of \cite[Theorem 2]{Konyagin}. \newline
First, let $m>2$ be an integer and define 
\begin{equation*} 
g(z)=\sum_{j=0}^m b_jz^j \text{ such that } b_m=1
\end{equation*}
and
\begin{equation*}
M(g)=\prod_{j=1}^m \max(1,|z_j|),
\end{equation*}
where $z_j$ are roots of $g$ counted with multiplicity for $1\leq j \leq m$. \newline
We know that
\begin{equation*}
\log M(g)=\prod_{j=1}^m \frac1{2\pi} \int_0^{2\pi} \log|g(e^{i\phi})|d\phi
\end{equation*}
So, by Jensen's inequality
\begin{equation}\label{eq: upper bd of M}
1 \leq M\left(\sum_{j=0}^m b_jz^j\right) \leq \sum_{j=0}^m |b_j|.
\end{equation}
If $M(g)=1$ then by Kronecker's theorem \cite{kronecker} all $z_j$ are roots of unity. Otherwise from \cite{Dobrowolski} we get 
\begin{equation*}
\exp(\lambda_m) \leq M(g), \textit{ where } \lambda_m = c\left(\frac{\log \log m}{\log m}\right)^3.    
\end{equation*}
Now we will consider the case where our polynomial $f$ is divisible by a noncyclotomic polynomial $g$ such that $\deg(g)=m\leq m_0$. From \cite[Lemma 3]{Dobrowolski} we know that there exists a prime number $p$ such that ${\log(2Kd+1)}\frac{1}{\lambda_{m_0}}< p < 2\log(2Kd+1)\frac {1}{\lambda_{m_0}}$ and such that all $m$ roots of $g$ raised to the power of $p$ are algebraic numbers of degree $m$. Note, this also implies that they are distinct.
Then, if $g_p(w)=\prod^m_{j=1} (w-z^p_j)$ we have
\begin{equation}\label{eq: lower bd of M}
M(g_p)=M(g)^p \geq \exp(p\lambda_{m_0}) > 2Kd+1.
\end{equation}
Suppose that $g(z)$ divides both $f_1$ and $f_2$, which are two distinct polynomials from $\mathcal{P}_{d,K}$ and 
\begin{equation*}
f_1(z)-f_2(z)=\sum^{ [d/p]}_{j=0}a_jz^{jp}=h(z^p).
\end{equation*}

We know that coefficients of the polynomial h are in $[-2K;2K]$. Thus, inequality \eqref{eq: upper bd of M} implies that $M(h)\leq 2Kd+1$. On the other hand, every root $z_j^p$ of $g_p$ is also a root of $h$. Hence, $g_p$ divides $h$ and by the lower bound on $M(g)$ as in \eqref{eq: lower bd of M} we can see that $M(h)>2Kd+1$, which gives us a contradiction. Therefore, if $f \in \mathcal{P}_{d,K}$ is divisible by $g$ then $f$ is uniquely determined by its coefficients $a_j$, where $j \neq 0 (\text{mod } p)$. \newline
Thus,
\begin{align}
\# \{ f \in \mathcal{P}_{d,K}: g\mid f\}&\leq\frac{(2K+1)^d}{(2K+1)^{d/p}} \\
&\leq\frac{(2K+1)^d}{(2K+1)^{\,d\,\lambda_{m_0}/2\log(Kd+1)}} \notag\\
&\leq(2K+1)^d\exp \left(\frac{-\log(2K+1)\,d\,\lambda_{m_0}}{2\log Kd} \right).\notag
\end{align}

To estimate the probability that polynomials $g$ of $\deg g \leq m_0$ divides at least one $f \in \mathcal{P}_{d,K},$ we consider any such polynomial
\begin{equation*}
g(z)=\sum_{j=0}^m b_jz^j=\prod^{m}_{j=1}(z-z_j).
\end{equation*}
Since $|z_j|<K+1$ for every $j$, representing the coefficients of the polynomial g as symmetric polynomials of its zeros, we find $|b_j|\leq K^{m-j} {m\choose j} <(2K)^m$.  \newline
Thus,
\begin{equation*}
\# \{ g: \exists f \in \mathcal{P}_{d,K} \text{ such that } g\mid f\}\leq\prod^{m_0}_{j=0}(2(2K)^{m_0}+1)\leq(2K+1)^{{m_0}^2}.
\end{equation*}
Recall that $\lambda_m = c\left(\frac{\log \log m}{\log m}\right)^3$ and $m_0=\frac{\sqrt{d}}{(\log(d))^2}$, hence the following inequality hold for sufficiently large $d$:
\begin{align}\label{eq: gl}
    \lambda_{m_{0}}= c\left(\frac{\log \log (\sqrt{d}/(\log d)^2)}{\log (\sqrt{d}/(\log d)^2)}\right)^3\leq
    8c\left(\frac{\log(\log({d}^{1/4})}{\log d}\right)^3
\end{align}
Using these inequalities we derive upper bound on probability of interest
\begin{align}\label{eq:final of 2.2}
&\mathbb{P} \left( f \in \mathcal{P}_{d,K} \text{is divisible by at least one noncyclotomic polynomial of degree} \leq m_0\right)\\
&\qquad\qquad\leq (2K+1)^{{m_0}^2}\exp \left( \frac{-\log(2K+1)\,d\,\lambda_{m_0}}{2\log Kd}\right) \notag\\
&\qquad\qquad=\exp\left(d\,\log(2K+1)\left(\frac{1}{(\log d)^4}-\frac{\lambda_{m_0}}{2\log Kd}\right)\right)\notag\\
&\qquad\qquad\leq\exp\left(d\,\log(2K+1)\left(\frac{1}{(\log d)^4}-\frac{c}{a}\cdot\frac{\lambda_{m_0}}{\log d}\right)\right)\notag
\end{align}
Where we used that $K\leq d^{a-1}$ for some natural number $a>1$. Using the inequality in~\eqref{eq: gl} we get
\begin{align*}
&\qquad\leq \exp\left( \frac{d\log(2K+1)}{\log d}\left(\frac{1}{(\log d)^3}- \frac{8c}{a}\left(\frac{\log(\log({d}^{1/4})}{\log d}\right)^3\right)\right)\notag\\
&\qquad\leq \exp\left( {d\log(2K+1)}\left(\frac{- c(\log (\log({d}^{1/4})))^3}{(\log d)^4}\right)\right)\notag\\
&\qquad\leq \exp\left( {d\log(2K+1)}\left(\frac{- C}{(\log d)^4}\right)\right)\notag\\
&\qquad\leq (2K+1)^{\left(-\frac{Cd}{(\log d)^4}\right)},\notag
\end{align*}
where these inequalities hold for sufficiently large $d$ and some positive constant $C$. This completes the proof. 
\end{proof}
Proposition~\ref{pr2.1} and~\ref{pr2.2} put together give the proof of Theorem~\ref{th:main}.
\begin{proof}[Proof of Theorem~\ref{th:main}]
Statement of Theorem~\ref{th:main} follows directly from Proposition \ref{pr2.1}, Proposition \ref{pr2.2}, and the fact that
\begin{equation*}\label{eq: connection}
    O\left((2K+1)^{\left(\frac{- Cd}{(\log d)^4}\right)}\right) \leq \min\left\{O\left(\sqrt{\frac{K}{d}}\right);O\left(\frac{1}{K}\right)\right\}.
\end{equation*}
\end{proof}

Theorem~\ref{th:main2} follows from the analogous argument to the one above with a slight modification in the non-cyclotomic case.

\begin{proof}[Proof of Theorem~\ref{th:main2}]
Since $$m_0=\frac{\sqrt{d}}{\log^2(d)}\geq\frac{\sqrt{d}}{\log^2(Kd)}=m_1$$ we know that probability that $f\in\mathcal{P}_{d,K}$ is divisible by a cyclotomic polynomial up to degree $m_0$ is greater or equal to the probability that it is divisible by such a polynomial of degree $m_1$. So by Proposition \ref{pr2.1} we know that
\begin{equation*}
    \mathbb{P} (f \in \mathcal{P}_{d,K} \text{ is divisible by at least one cyclotomic polynomial of degree up to } m_1 )\leq O\left(\sqrt{\frac{K}{d}}\right)
\end{equation*}
and if $e^{(d^{0.25}/b)}/d> K\geq cd$, then
\begin{equation*}
    \mathbb{P} (f \in \mathcal{P}_{d,K} \text{ is divisible by at least one cyclotomic polynomial of degree up to } m_1 )\leq O\left(\frac1{K}\right).
\end{equation*}
Now to prove analog of Proposition \ref{pr2.2} we need to make sure that $\lambda_{m_1}$ is well defined. In fact, this is what defines an upped bound of $e^{(d^{0.25}/b)}/d>K(d)$, where $b:=\exp\left(e^{1/\sqrt[3]{4c_1 }}/2\right)$ for some constant $c_1$. By the same argument as in the proof of Proposition \ref{pr2.2} up until the equation \eqref{eq:final of 2.2} we get that
\begin{align*}\label{eq:final of 2.8}
&\mathbb{P} \left( f \in \mathcal{P}_{d,K} \text{is divisible by at least one noncyclotomic polynomial of degree} \leq m_1\right)\\
&\qquad\qquad\leq (2K+1)^{{m_1}^2}\exp \left( \frac{-\log(2K+1)d\lambda_{m_1}}{2\log (Kd)}\right) \notag\\
&\qquad\qquad=\exp\left(\log(2K+1)\left(\frac{d}{(\log (Kd))^4}-\frac{d\lambda_{m_1}}{2\log (Kd)}\right)\right)\notag\\
&\qquad\qquad=\exp\left(\frac{\log(2K+1)d}{\log(Kd)}\left(\frac{1}{(\log (Kd))^3}-\frac{\lambda_{m_1}}{2}\right)\right) \notag
\end{align*}
Now, since $K<e^{(d^{0.25}/b)}/d$, where $b=\exp(e^{1/\sqrt[3]{4c}}/2)$, where $c$ is the constant in the definition of $\lambda_m$. Note that $\log(m_1)>1$ and $(\log\log(m_1))^3>\frac1{4c}$ for sufficiently large $d$, so we get that
\begin{align*}
&\qquad\qquad\leq \exp\left( \frac{d\log(2K+1)}{\log (Kd)}\left(\frac{1}{(\log (Kd))^3}- \frac{c}{2}\left(\frac{\log \log (m_1)}{\log (m_1)}\right)^3\right)\right)\\
&\qquad\qquad\leq \exp\left( \frac{d\log(2K+1)}{\log (Kd)}\left(\frac{1}{(\log (Kd))^3}- \frac{8c}{2}\left(\frac{\log\log(m_1)}{\log d }\right)^3\right)\right)\notag\\
&\qquad\qquad\leq \exp\left( \frac{d\log(2K+1)}{\log (Kd)}\left(\frac{1}{(\log (Kd))^3}- 4c\frac{(\log\log(m_1))^3}{(\log (Kd))^3}\right)\right)
\end{align*}
Using that for $d$ large enough $(\log\log(m_1))^3>\frac1{4c}$ we have that for some some constant $C>0$ 
\begin{align*}
&\qquad\qquad\leq \exp\left( {d\log(2K+1)}\left(\frac{- C}{(\log (Kd))^4}\right)\right)\\
&\qquad\qquad= (2K+1)^{\left(\frac{- Cd}{(\log (Kd))^4}\right)}\notag\\
&\qquad\qquad\leq \min\left\{O\left(\sqrt{\frac{K}{d}}\right);O\left(\frac{1}{K}\right)\right\} \quad \text{as } d\to\infty.
\end{align*}
This implies the desired result.
\end{proof}
%
%
\section{Hilbert's irreducibility theorem with different parameters}\label{Hilbert th}
In this section we present an alternative way of studying irreducibility of polynomials as both range of coefficients and degree tend to infinity. This approach based on a proof of classical Hilbert's irreducibility theorem that I learned from Philip Matchett Wood in private conversations \cite{wood}.  
\begin{thm}\label{th3.1} Let $f(x) \in \mathcal{P}_{d,K}$. We have $\mathbb{P}(f(x) \text{ is irreducible }) \to 1$  as $d$ and $K \to \infty$ as long as $$d \leq \log_2(K)-\log_2(1+(\log K)^2).$$
\end{thm}
\begin{proof} We first estimate the number of reducible polynomials in $\mathcal{P}_{d,K}.$ By Bertrand’s Postulate, there is $p,$ which is a prime number, such that $2K+1\leq p\leq4K.$ Let $N_k$ be the number of reducible polynomials $f(x)$, such that $f(0)=k$ for each $k\in[-K,K].$ Since there are at most $(2K+1)^{d-1}$ polynomials with constant coefficient $0$ we know that
$$N_0\leq(2K+1)^{d-1}.$$
For $k\neq0$, there are $a,b>0$ such that $a+b=d$ and
\begin{equation}\label{eq:poly}
    f(x)=(x^a+n_{a-1}x^{a-1}+\cdots+n_1x+k_1)(x^b+m_{b-1}x^{b-1}+\cdots+m_1x+k_2).
\end{equation}
From the fact that the product of $k_1$ and $k_2$ is equal to $k$ it follows that the number of such pairs $(k_1,k_2)$ is at most $2\tau(k)$, where $\tau(k)$ is the number of positive divisors of $k.$ We also notice that polynomials factor in the same way as above viewing \eqref{eq:poly} modulo $p$. Since we chose $p\geq 2K+1$, then any distinct factorizations of $f(x)$ over the integers would remain distinct modulo $p$. Thus, when we estimate $N_k$ for $k\neq0$, we can consider factorization of polynomials modulo $p$ instead of over the integers. Modulo $p,$ there are at most $p^{a-1}$ possible coefficients $n_i$ and $p^{b-1}$ possible coefficients $m_i$. Thus,
\begin{align*}
    \#(\text{reducible polynomials in }\mathcal{P}_{d,K})&=\sum_{k=-K}^KN_k\\
    &\leq(2K+1)^{d-1}+2\sum_{k=1}^Kp^{a+b-2}2\tau(k)\notag\\
    &\leq(4K)^{d-2}\left(4K+4\sum^K_{k=1}\tau(k)\right).
\end{align*}
Now by a theorem proved by Dirichlet \cite[Theorem 3.3]{Apostol} we have that
\begin{equation*}
\sum^K_{k=1}\tau(k)=K\log K+O(K),
\end{equation*}
hence
\begin{align*}
\mathbb{P}(f(x) \text{is reducible})&\leq \frac{(4K)^{d-2}(4K+O(K\log K))}{(2K+1)^d} \\
&\leq \frac{(4K)^{d-1}}{(2K)^d} +O\left(\frac{4^{d-2}K^{d-1}\log K}{(2K)^d}\right) \notag\\
&\leq O\left(\frac{2^{d-2}}{K} + \frac{2^{d-4}\log K}{K}\right)\notag\\
& \leq O \left(\frac{2^{d}(1+\log K)}{K}\right)\notag
\end{align*}
which goes to zero as $ K \to \infty$ as long as $d= o\left(\log_2K-\log_2(1+\log K)\right).$
This completes the proof.
\end{proof}

\begin{thm}\label{th3.2}Let $f_d(x)=x^d+a_{d-1}x^{d-1}+a_{d-2}x^{d-2}+\ldots+a_1x+a_0$ be a random polynomial of degree $d$ with integer coefficients $a_i$, where $a_i$ are i.i.d. centered binomial random variables with parameters $4^{2d}$ and $p=\frac12$, then $$\mathbb{P}(f_d (x) \text{ is irreducible over }\mathbb{Z}) \to 1\quad \text{as}\quad  d \to \infty$$
\end{thm}

\begin{proof}
First, we fix $b_0,b_1,\dots,b_{n-1}\in\{-K,-K+1,\ldots, K\}$. Polynomial $f_d(x)$ has $b_i$'s as its coefficients with the following probability
\begin{align}
\mathbb{P}\left(a_i=b_i \text{ for each } 0\leq i\leq n-1\right)&\leq \mathbb{P}\left(a_i=0 \text{ for each } 0\leq i\leq n-1\right)\notag \\ 
&= \prod_{i=0}^n \mathbb{P}(a_i=0) = \mathbb{P}(a_i=0)^d \notag\\
&=\left(\frac{1}{4^d\sqrt{2\pi}}\right)^d\notag
\end{align}
Therefore, the probability that $f_d$ is reducible can be calculated by dividing it in two cases: when for all $i\leq d-1$ $a_i \leq 4^d$ and otherwise. So we can rewrite the probability of interest as
\begin{equation*}
\mathbb{P}\left(f_d (x) \text{ is reducible}\right)= \left(\text{\# of reducible polynomials}\right)(A+B),
\end{equation*}
where
\begin{equation*}
    A=\mathbb{P}\left(a_i=b_i \text{ for each } 0\leq i\leq d-1 \text{ and } |a_i| \leq 4^d\right)
\end{equation*}
and
\begin{equation*}
    B=\mathbb{P}\left(a_i=b_i \text{ for each } 0\leq i\leq d-1 \text{ and } \exists \text{ i such that } |a_i| > 4^d \right).
\end{equation*}
By Hoeffding inequality for all $i$ $$\mathbb{P}\left(|a_i| > 4^d \right)\leq 2\exp\left(-2\frac{4^d}{d}\right),$$
so by union bound
\begin{equation}\label{eq: hoef}
    \mathbb{P}\left(|a_i| > 4^d \text{ for some $i$}\right)\leq 2d\cdot\exp\left(-2\frac{4^d}{d}\right).
\end{equation}
Now, from the argument in Theorem \ref{th3.1} follows that the number of reducible polynomials can be bounded by $(4K)^{d-2}(4K+O(K\log K))$, where $\vert a_i \vert \leq K.$
Since the bound in~\eqref{eq: hoef} decays to $0$ as $d\exp(-4^d/d)$ and $K\leq 4^{2d}$ we conclude that $$\left(\text{\# of reducible polynomials}\right)\cdot B=o(1)$$ 
To estimate the rest we set $K=4^d.$ By simplifying the expression above we get that 
\begin{equation*}
(\# \text{ of reducible polynomials}) \leq O \left(4^{d-1}\left(4^d\right)^{d-1}+4^{d-2}\left(4^d\right)^{d-1}\log\left(4^d\right)\right).    
\end{equation*} 
Thus, 
\begin{align*}
\mathbb{P}(f_d (x) \text{ is reducible}) &=\mathbb{P}\left(a_i=b_i \text{ for each } 0\leq i\leq n-1  \text{ given that } |a_i| \leq 4^d \right)\notag\\
&\qquad\qquad\times(\text{\# of reducible polynomials}) + o(1) \notag\\
&= C\left(\frac{4^{d-2}4^{d(d-1)}}{\left(\sqrt{2\pi}4^d\right)^d}\right)\left(4+\log\left(4^d\right)\right) + o(1) \notag\\
&\leq C\left(\frac{2^{2d-4}4^{d(d-1)}}{2^d 4^{d^2}}\right)\left(4+\log\left(4^d\right)\right) + o(1) \notag\\
&\leq C\left( \frac{2^{d-4}\log\left(4^d\right)}{4^d}\right)+o(1) \to 0 \textit{ as d} \to \infty.\notag
\end{align*}
Therefore, $\mathbb{P}\left(f_d (x) \text{ is reducible}\right) \to 0$ as $n$ $\to \infty.$
\end{proof}
%
%
\appendix
\section{}\label{Appendix}
Recall that, as in \cite{bourgain}, we define a random variable to be $p$-bounded of exponent $r$ in the following way:

\begin{defn}\label{4.1} Let $p,q \in \mathbb{R}_{+}$ such that $0 <q< p< 1$ and let $r \in \mathbb{Z}_{+}$. A $\mathbb{Z}$-valued random variable $a$  is called $p \textit{-bounded of exponent }r$ if there exists a $\mathbb{Z}$-valued symmetric random variable $\beta^{(\mu)}$ taking value 0 with probability $1-\mu=p$ such that the following conditions holds:  
\begin{enumerate}
\item $\max_x \mathbb{P}(a = x) \leq p$ 
\item $q \leq \min_x \mathbb{P}(\beta^{(\mu)}=x) \text{ and } \max_x \mathbb{P}(\beta^{(\mu)}=x) \leq p$
\item $\forall t \in \mathbb{R} \text{ we have } \left|\mathbb{E}(\exp(2\pi i a t))\right|^r \leq \mathbb{E}\left(\exp(2\pi i \beta^{(\mu)} t)\right).$ 
\end{enumerate}
\end{defn}
The third condition can be rewritten as 
\begin{equation*}
|\mathbb{E}(\exp(2\pi i a t))|^r \leq \mathbb{E}(\exp(2\pi i \beta^{(\mu)} t))=1-\mu+\mu \sum_{s} p_s\cos(2\pi b_s t),    
\end{equation*} where $p_s= \mathbb{P}(\beta^{(\mu)}= b_s)/ \mu=\mathbb{P}(\beta^{(\mu)}= -b_s)/ \mu$.
To proof of Lemma~\ref{lem: pbdd} follows from straight forward computations
\begin{proof}[Proof of Lemma~\ref{lem: pbdd}]
Set $p_0=\frac{1}{2K+1}.$ By the symmetry of $A$ we have 
\begin{equation*}
\left|\mathbb{E}\exp(2\pi i At)\right|= \left| p_0 + \frac{2}{2K+1}\sum_{j=1}^K\cos(2\pi j t)\right|.
\end{equation*}
Thus,
\begin{align*}
\left|\mathbb{E}\exp(2\pi i At)\right|^2&=\left(p_0 + \frac{2}{2K+1}\sum_{j=1}^K\cos(2\pi j t)\right)^2 \notag\\
&=p_0^2+\frac{4 p_0}{2K+1}\sum_{j=1}^K\cos(2\pi j t)+\frac{4}{(2K+1)^2}\left(\sum_{j=1}^K\cos(2\pi j t)\right)^2.
\end{align*}
Converting the product of cosines into the sums we get
\begin{align*}
&=p_0^2+\frac{4p_0}{2K+1}\sum_{j=1}^K\cos(2\pi j t)+\frac{4}{(2K+1)^2}\sum_{j=1}^K\cos^2(2\pi j t)+\frac{8}{(2K+1)^2}\sum_{j<i}^K\cos(2\pi j t)\cos(2\pi i t )\notag\\
&=p_0^2+\frac{4p_0}{2K+1}\sum_{j=1}^K\cos(2\pi j t)+\frac{2}{(2K+1)^2}+\frac{2}{(2K+1)^2}\sum_{j=1}^K\cos(4\pi j t) \notag\\
&\qquad\quad+\frac{4}{(2K+1)^2}\sum_{j<i}^K\left(\cos(2\pi (i-j) t)+\cos(2\pi (j+i) t)\right). \notag\\
\end{align*}
Finally grouping the terms based in the cosines we derive that
\begin{align*}
&=p_0^2+\frac{2}{(2K+1)^2}+\frac{4p_0}{2K+1}\sum_{j=1}^K\cos(2\pi j t)+\frac{2}{(2K+1)^2}\sum_{j=1}^K\cos(4\pi j t)\notag\\
&\qquad\quad+\frac{4}{(2K+1)^2}\sum_{j<i}^K\cos(2\pi (i-j) t) +\frac{4}{(2K+1)^2}\sum_{j<i}^K\cos(2\pi (j+i) t)\notag\\
&=p_0^2+\frac{2}{(2K+1)^2}+\left(\frac{4p_0}{2K+1}+\frac{4}{(2K+1)^2}(K-1)\right)\cos(2\pi t)\notag\\
&\qquad\quad+\left(\frac{4p_0}{2K+1}+\frac{4}{(2K+1)^2}(K-2)+\frac{2}{(2K+1)^2}\right)\cos(4\pi t)\notag\\
&\qquad\quad+\left(\frac{4p_0}{2K+1}+\frac{4}{(2K+1)^2}(K-3)+\frac{4}{(2K+1)^2}\right)\cos(6\pi t)\notag\\
&\qquad\quad+\left(\frac{4p_0}{2K+1}+\frac{4}{(2K+1)^2}(K-4)+\frac{2}{(2K+1)^2}+\frac{4}{(2K+1)^2}\right)\cos(8\pi t) \notag\\
&\qquad\quad\dots \notag\\
&\qquad\quad+\left(\frac{2}{(2K+1)^2}\right)\cos(4K\pi t).\notag
\end{align*}
Hence,

\begin{equation*}
|\mathbb{E}\exp(2\pi i At)|^2 \leq 1 - \frac{4Kp_0}{2K+1}+\left(\frac{4p_0}{2K+1}+\frac{4}{(2K+1)^2}(K-1)\right)\cos(2\pi t).
\end{equation*}

Thus, for $p_0=\frac{1}{2K+1}$ $\beta^{(\mu)}$ can be constructed as follows:
\begin{equation}
\beta^{(\mu)}=
\begin{cases}
1, & \text{with probability}\  \frac{2K}{(2K+1)^2} \\
0, & \text{with probability}\  1-\frac{4K}{(2K+1)^2}\\
-1, & \text{with probability}\  \frac{2K}{(2K+1)^2}\\
\end{cases}
\end{equation}
Therefore A is $\left(1-\frac{4K}{(2K+1)^2}\right)$-bounded of exponent 2 with $q= \frac{2K}{(2K+1)^2}$,
\end{proof}

\begin{proof}[Proof of Lemma~\ref{lem: LWO}]
For each $k$ denote $\beta_k^{(\mu)}$ as symmetric random variable from Definition \ref{def: pbdd} that correspond to $a_k$. Also suppose $Q$ a sufficiently large number. Finally, to simplify the notation, let us denote $\sum^m_{k=0}{a_k}=X_m.$ 

\begin{align}
\mathbb{P}\left(X_{m}=x\right)&=\mathbb{E}\mathds{1}_{\{X_{m}=x\}}\\
&=\frac{1}{Q}\mathbb{E}\sum_{\xi \in \mathbb{Z}/Q\mathbb{Z}}\exp(2\pi i (X_{m}-x)\xi /Q)) \notag\\
&\leq \frac{1}{Q}\sum_{\xi \in \mathbb{Z}/Q\mathbb{Z}}\prod_{k=0}^{m} |\mathbb{E}\exp(2\pi i a_k\xi /Q)|\notag\\
&\leq \prod_{k=0}^{m}\left(\frac{1}{Q}\sum_{\xi \in \mathbb{Z}/Q\mathbb{Z}} |\mathbb{E}\exp(2\pi i a_k\xi /Q)|^m\right)^{1/m} & \mbox{ (by H\"older's inequality)}\notag\\
&\leq \frac{1}{Q}\sum_{\xi \in \mathbb{Z}/Q\mathbb{Z}} |\mathbb{E}\exp(2\pi i a_{k_0}\xi /Q)|^m,\notag
\end{align}
where $k_0$ corresponds to the largest factor in the previous line. Since $\{a_k\}$ are identical, we know $\mathbb{P}(\beta_{k_0}^{(\mu)}=1)=\mu p_{k_0}/2$. Thus, since $\mu p_{k_0}\geq 2q,$ by applying the inequality from \cite[Corollary 7.13]{Tao} we get the following:
\begin{align}
&\leq \frac{1}{Q}\sum_{\xi \in \mathbb{Z}/Q\mathbb{Z}} \Big{(} 1-2q +2q\cos(2\pi \xi/Q)\Big{)}^{m/r}\\
&\leq \int_0^1 \Big{(}1-2q+2q\cos(2\pi t)\Big{)}^{m/r}dt+\frac{1}{m}\notag\\
&=\frac{C\sqrt{r}}{\sqrt{qm}}.\notag
\end{align}
\end{proof}


\bibliography{poly.bib}

\begin{bibdiv}
\begin{biblist}

\bib{Apostol}{book}{
      author={Apostol, Tom~M.},
       title={Introduction to analytic number theory},
      series={Undergraduate Texts in Mathematics},
   publisher={Springer-Verlag, New York-Heidelberg},
        date={1976},
      review={\MR{0434929}},
}

\bib{BS_K_K}{misc}{
      author={Bary-Soroker, Lior},
      author={Koukoulopoulos, Dimitris},
      author={Kozma, Gady},
       title={Irreducibility of random polynomials: general measures},
        note={https://arxiv.org/abs/2007.14567},
}

\bib{Bary-Soroker_Kozma}{article}{
      author={Bary-Soroker, Lior},
      author={Kozma, Gady},
       title={Irreducible polynomials of bounded height},
        date={2020},
        ISSN={0012-7094},
     journal={Duke Math. J.},
      volume={169},
      number={4},
       pages={579\ndash 598},
  url={https://doi-org.proxy2.library.illinois.edu/10.1215/00127094-2019-0047},
      review={\MR{4072635}},
}

\bib{BBBSWW}{article}{
      author={Borst, Christian},
      author={Boyd, Evan},
      author={Brekken, Claire},
      author={Solberg, Samantha},
      author={Wood, Melanie~Matchett},
      author={Wood, Philip~Matchett},
       title={Irreducibility of random polynomials},
        date={2018},
        ISSN={1058-6458},
     journal={Exp. Math.},
      volume={27},
      number={4},
       pages={498\ndash 506},
      review={\MR{3894729}},
}

\bib{bourgain}{article}{
      author={Bourgain, Jean},
      author={Vu, Van~H.},
      author={Wood, Philip~Matchett},
       title={On the singularity probability of discrete random matrices},
        date={2010},
        ISSN={0022-1236},
     journal={J. Funct. Anal.},
      volume={258},
      number={2},
       pages={559\ndash 603},
      review={\MR{2557947}},
}

\bib{Breullard}{article}{
      author={Breuillard, Emmanuel},
      author={Varj\'{u}, P\'{e}ter~P.},
       title={Irreducibility of random polynomials of large degree},
        date={2019},
        ISSN={0001-5962},
     journal={Acta Math.},
      volume={223},
      number={2},
       pages={195\ndash 249},
      review={\MR{4047924}},
}

\bib{Dobrowolski}{article}{
      author={Dobrowolski, E.},
       title={On a question of {L}ehmer and the number of irreducible factors
  of a polynomial},
        date={1979},
        ISSN={0065-1036},
     journal={Acta Arith.},
      volume={34},
      number={4},
       pages={391\ndash 401},
      review={\MR{543210}},
}

\bib{Konyagin}{article}{
      author={Konyagin, S.~V.},
       title={On the number of irreducible polynomials with {$0,1$}
  coefficients},
        date={1999},
        ISSN={0065-1036},
     journal={Acta Arith.},
      volume={88},
      number={4},
       pages={333\ndash 350},
      review={\MR{1690454}},
}

\bib{kronecker}{article}{
      author={Kronecker, L.},
       title={Zwei {S}\"{a}tze \"{u}ber {G}leichungen mit ganzzahligen
  {C}oefficienten},
        date={1857},
        ISSN={0075-4102},
     journal={J. Reine Angew. Math.},
      volume={53},
       pages={173\ndash 175},
      review={\MR{1578994}},
}

\bib{odlyzko}{article}{
      author={Odlyzko, A.~M.},
      author={Poonen, B.},
       title={Zeros of polynomials with {$0,1$} coefficients},
        date={1993},
        ISSN={0013-8584},
     journal={Enseign. Math. (2)},
      volume={39},
      number={3-4},
       pages={317\ndash 348},
      review={\MR{1252071}},
}

\bib{lowdeg}{article}{
      author={O'Rourke, Sean},
      author={Wood, Philip~Matchett},
       title={Low-degree factors of random polynomials},
        date={2019},
        ISSN={0894-9840},
     journal={J. Theoret. Probab.},
      volume={32},
      number={2},
       pages={1076\ndash 1104},
      review={\MR{3959638}},
}

\bib{Prachar}{book}{
      author={Prachar, Karl},
       title={Primzahlverteilung},
      series={Grundlehren der Mathematischen Wissenschaften [Fundamental
  Principles of Mathematical Sciences]},
   publisher={Springer-Verlag, Berlin-New York},
        date={1978},
      volume={91},
        ISBN={3-540-08558-0},
        note={Reprint of the 1957 original},
      review={\MR{516660}},
}

\bib{Tao}{book}{
      author={Tao, Terence},
      author={Vu, Van},
       title={Additive combinatorics},
      series={Cambridge Studies in Advanced Mathematics},
   publisher={Cambridge University Press, Cambridge},
        date={2006},
      volume={105},
        ISBN={978-0-521-85386-6; 0-521-85386-9},
      review={\MR{2289012}},
}

\bib{wood}{misc}{
      author={Wood, P.~M.},
        date={2017},
        note={private communication},
}

\end{biblist}
\end{bibdiv}
\end{document}